\documentclass[12pt,bezier]{article}
\usepackage{amsmath,amssymb,amsfonts,xspace}
\usepackage{graphicx}

\newtheorem{prethm}{{\bf Theorem}}

\newenvironment{thm}{\begin{prethm}{\hspace{-0.5
               em}{\bf.}}}{\end{prethm}}

\newtheorem{prepro}{{\bf Proposition}}

\newtheorem{precor}{{\bf Corollary}}

\newenvironment{cor}{\begin{precor}{\hspace{-0.5
               em}{\bf.}}}{\end{precor}}

\newtheorem{preconj}{{\bf Conjecture}}

\newtheorem{preremark}{{\bf Remark}}

\newenvironment{remark}{\begin{preremark}\rm{\hspace{-0.5
               em}{\bf.}}}{\end{preremark}}

\newtheorem{prelem}{{\bf Lemma}}

\newenvironment{lem}{\begin{prelem}{\hspace{-0.5
               em}{\bf.}}}{\end{prelem}}

\newtheorem{preproof}{{\bf Proof.}}

\newenvironment{proof}[1]{\begin{preproof}{\rm
               #1}\hfill{$\Box$}}{\end{preproof}}

\newcommand{\od}{{\rm ord}}
\renewcommand{\thefootnote}

\title{Cycles are determined by their domination polynomials}
\author{Saieed Akbari$^{\,\rm a,b}$, ~ Mohammad Reza Oboudi$^{\,\rm b}$ \\
{\footnotesize {\em $^{\rm a}$School of Mathematics, Institute for Research in Fundamental Sciences (IPM),}}\\
 {\footnotesize {\em P.O. Box 19395-5746, Tehran, Iran}}\\
{\footnotesize {\em $^{\rm b}$Department of Mathematical Sciences, Sharif University of Technology,}}\\
{\footnotesize {\em P.O. Box 11155-9415,
 Tehran, Iran}}}
\begin{document}
\footnotetext{{\em E-mail Addresses}: {\tt s\_akbari@sharif.edu}
(S. Akbari), {\tt m\_r\_oboudi@math.sharif.edu} (M.R. Oboudi).}
\date{}
\maketitle

\begin{abstract}
Let $G$ be a simple graph of order $n$. A dominating set of $G$
is a set $S$ of vertices of $G$ so that every vertex of $G$ is
either in $S$ or adjacent to a vertex in $S$. The domination
polynomial of  $G$  is the polynomial $D(G,x)=\sum_{i=1}^{n}
d(G,i) x^{i}$,
 where $d(G,i)$ is the number of dominating sets  of $G$ of size
 $i$. In this paper we show that cycles are determined by their domination polynomials.

\end{abstract}

\noindent{\small\noindent{\small {\it AMS Classification}}:
05C38, 05C69.

\noindent\vspace{0.1mm}{\it Keywords}: Cycle; Dominating set;
Domination polynomial; Equivalence class.}

\section{Introduction}

Throughout this paper we will consider only simple graphs. Let
$G=(V,E)$ be a simple graph. The {\it order} of $G$ denotes the
number of vertices of $G$. For every vertex $v\in V$, the {\it
closed neighborhood} of $v$ is the set $N[v]=\{u \in V|uv\in E\}\cup
\{v\}$.
 For a set $S\subseteq V$, the closed neighborhood of $S$ is $N[S]=\bigcup_{v\in S}
 N[v]$. A set $S\subseteq V$ is a {\it dominating set} if $N[S]=V$, or
equivalently, every vertex in $V\backslash S$ is adjacent to at
least one vertex in $S$.
 The {\it domination number} $\gamma(G)$ is the minimum cardinality of a dominating set in $G$.
 A dominating set with cardinality $\gamma(G)$ is called a {\it $\gamma$-set}.
  For a detailed treatment of this parameter, the reader is referred to~\cite{hhs}.
 Let ${\cal D}(G,i)$ be the family of dominating sets of a graph
$G$ with cardinality $i$ and let $d(G,i)=|{\cal D}(G,i)|$.
 The {\it domination polynomial} $D(G,x)$ of $G$ is defined as
$D(G,x)=\sum_{i=1}^{|V(G)|} d(G,i) x^{i}$, ( see \cite{a}). Two
graphs $G$ and $H$ are said to be  {\it${\cal D}$-equivalent},
written $G\sim H$, if $D(G,x)=D(H,x)$. The {\it${\cal
D}$-equivalence class} of $G$ is defined as $[G]=\{H: H\sim G\}$. A
graph $G$ is said to be {\it${\cal D}$-unique}, if $[G]=\{G\}$. For
two graphs $G_1=(V_1,E_1)$ and $G_2=(V_2,E_2)$, {\it join} of $G_1$
and $G_2$ denoted by $G_1\vee G_2$ is a graph with vertex set
$V(G_1)\cup V(G_2)$ and edge set $E(G_1)\cup E(G_2)\cup \{uv| u\in
V(G_1)$ and $v\in V(G_2)\}$. We denote the complete graph of order
$n$, the cycle of order $n$, and the path of order $n$, by $K_n$,
$C_n$, and $P_n$, respectively. Also we denote $K_1\vee C_{n-1}$ by
$W_n$ and call it {\it wheel} of order $n$.

 Let $n\in \mathbb{Z}$ and $p$ be a prime
number. Then if $n$ is not zero, there is a nonnegative integer $a$
such that $p^a\mid n$ but $p^{a+1}\nmid n$, we let $\od_p\,n=a$.

In \cite{aap}, the following question was posed:
\begin{quote}
    {\it For every natural number n, $C_n$ is ${\cal D}$-unique}.
\end{quote}

Also in \cite{aap} it was proved that $P_n$ for $n\equiv 0\pmod3$
has ${\cal D}$-equivalence class of size two and it contains $P_n$
and the graph obtained  by adding two new vertices joined to two
adjacent vertices of $C_{n-2}$. In this paper we show that for every
positive integer $n$, $C_n$ is ${\cal D}$-unique.

\section{${\cal D}$-uniqueness of cycles}

In this section we will prove that $C_n$ is ${\cal D}$-unique.
This answers in affirmative a problem proposed in \cite{aap} on
${\cal D}$-equivalence class of $C_n$. As a consequence we obtain
that $W_n$ is ${\cal D}$-unique. We let $C_1=K_1$ and $C_2=K_2$.
We begin by the following lemmas.
\begin{lem} {\rm\cite{aap}}\label{regular}
Let $H$ be a $k$-regular graph with $N[u]\neq N[v]$, for every
$u,v\in V(H)$. If $D(G,x)=D(H,x)$, then $G$ is also a $k$-regular
graph.
\end{lem}

\begin{lem} {\rm\cite[Theorem 2.2.3]{a}}\label{union} If $G$ has $m$ connected components $G_1,\dots,G_m $.
Then $D(G,x)=\prod_{i=1}^m D(G_{n_i},x)$.
\end{lem}

 The next lemma gives a recursive formula for the determination of
 the domination polynomial of cycles.
\begin{lem}\label{cycle} {\rm\cite[Theorem 4.3.6]{a}} For every $n\geq 4$,
$$D(C_n,x)=x(D(C_{n-1},x)+D(C_{n-2},x)+D(C_{n-3},x)).$$
\end{lem}

\begin{lem}\label{Zhang} {\rm\cite[Theorem 1]{fhrv}} For every $n\geq 1$, $\gamma(C_n)=\lceil\frac{n}{3}\rceil$.
\end{lem}
\begin{lem}\label{-1} If $n$ is a positive integer and $\alpha_n:=D(C_n,-1)$, then the following holds:
$$\alpha_n=\left\{
  \begin{array}{ll}
 3, & \hbox {if~ $n\equiv0\pmod4$;}  \\
-1, & \hbox {otherwise}.
 \end{array}
 \right.$$
\end{lem}
\begin{proof}{By Lemma \ref{cycle}, for every $n\geq4$, $\alpha_n=-(\alpha_{n-1}+\alpha_{n-2}+\alpha_{n-3})$. Now, by induction on $n$ the proof is complete.}
\end{proof}

\begin{lem}\label{ord} For every positive integer $n$,
$$\od_3\,D(C_n,-3)=\left\{
  \begin{array}{ll}
 \lceil\frac{n}{3}\rceil+1, & \hbox {if~ $n\equiv0\pmod3$;}  \\
\lceil\frac{n}{3}\rceil \hbox{~or~} \lceil\frac{n}{3}\rceil+1, & \hbox {if ~$n\equiv1\pmod3$;} \\
 \lceil\frac{n}{3}\rceil, & \hbox {if ~$n\equiv2\pmod3$}.
  \end{array}
 \right.$$
\end{lem}
\begin{proof}{Let $a_n:=D(C_n,-3)$. By Lemma \ref{cycle}, for any $n\geq 4$, $a_n=-3(a_{n-1}+a_{n-2}+a_{n-3})$. Since
$D(C_1,x)=x$, $D(C_2,x)=x^2+2x$, and $D(C_3,x)=x^3+3x^2+3x$, one
has $a_1=-3$, $a_2=3$, $a_3=-9$. Now, by induction on $n$ one can
easily see that $\od_3\,a_n\ge \lceil\frac{n}{3}\rceil$. Suppose
that $a_n=(-1)^n3^{{\lceil\frac{n}{3}\rceil}}b_n$. By Lemma
\ref{cycle} it follows that for every $n$, $n\geq 4$ the following
hold,
\begin{equation}\label{bn}b_n=\left\{
  \begin{array}{ll}
 3b_{n-1}-3b_{n-2}+b_{n-3}, & \hbox {if~ $n\equiv0\pmod3$;} \\
b_{n-1}-b_{n-2}+b_{n-3}, & \hbox {if ~$n\equiv1\pmod3$;}\\
 3b_{n-1}-b_{n-2}+b_{n-3}, & \hbox {if ~$n\equiv2\pmod3$}.
  \end{array}
 \right.\end{equation}
It turns out that $b_i$, for every $i$, $1\leq i\leq 30$ modulo
9, are as follows:\\
1, 1, 3, 3, 7, 6, 2, 7, 3, 7, 7, 3, 3, 4, 6, 5, 4, 3, 4, 4, 3, 3,
1, 6, 8, 1, 3, 1, 1, 3.

So for every $n$, $1\leq n\leq 30$, $9\nmid b_n$. By (\ref{bn})
and induction on $t$, it is easily seen that for every $t$,
$t\geq1$, $b_{t+27}\equiv b_t\pmod9$. Hence for any positive
integer $n$, $9\nmid b_n$ or equivalently $\od_3\,b_n\leq1 $. By
induction on $n$, and using (\ref{bn}), we find that if $n\equiv
0\pmod3$, $\od_3\,b_n=1 $, and $\od_3\,b_n\in\{0,1\}$ for
$n\equiv1\pmod3$, and $\od_3\,b_n=0$ for $n\equiv2\pmod3$. This
completes the proof.}
\end{proof}

\begin{remark} Since for every $t$, $t\geq1$, $b_{t+27}\equiv b_t\pmod9$, by considering $b_n$ for $1\leq n\leq 30$,
 we conclude that in the  case $n\equiv 1\pmod3$,
$$\od_3\,a_n=\left\{
  \begin{array}{ll}
    \lceil\frac{n}{3}\rceil+1,& \hbox{if ~$ n\in \{4, 13, 22\}$}\, (mod\,27 ); \\
    \lceil\frac{n}{3}\rceil, & \hbox{otherwise.}
  \end{array}
\right.$$
\end{remark}

Now, we prove our main result.

\begin{thm}\label{theorem5} For every positive integer $n$, cycle $C_n$ is ${\cal
  D}$-unique.
\end{thm}
\begin{proof}{ The assertion is trivial for $n=1, 2, 3$. Now, let $n\geq 4$. Let $G$ be a simple graph with
$D(G,x)=D(C_n,x)$. By Lemma \ref{regular}, $G$ is 2-regular and
so it is a disjoint union of cycles
$C_{n_1},\ldots,C_{n_k}$. Hence, by Lemma \ref{union},
$D(G,x)=\prod_{i=1}^k D(C_{n_i},x)$. Thus $n=n_1+\cdots+n_k$ and
by Lemma \ref{Zhang},
${\lceil\frac{n}{3}\rceil}={\lceil\frac{n_1}{3}\rceil}+\cdots+{\lceil\frac{n_k}{3}\rceil}.$
Therefore at least $k-2$ numbers of $n_1, \ldots,n_k$,
 are divisible by 3. On the other hand,
$$\od_3\,D(C_n,-3)=\sum_{i=1}^k \od_3\,D(C_{n_i},-3).$$
Now, by Lemma \ref{ord} it is easily  seen that $k\leq3$.

Now, let $\alpha_n:=D(C_n,-1)$, for every positive integer $n$.
Since $D(C_n,x)=\prod_{i=1}^k D(C_{n_i},x)$, therefore
$\alpha_n=\prod_{i=1}^k \alpha_{n_i}$. By Lemma \ref{-1},
$\alpha_n\in\{-1,3\}$. If $\alpha_n=3$, then only one of the
numbers $n_1, \ldots, n_k$ is divisible by 4, and therefore $k$
is an odd number. If $\alpha_n=-1$, then for every $i$, $1\leq
i\leq k$, $\alpha_{n_i}=-1$, and thus $k$ is an odd number. Since $k\leq3$, then $k\in\{1,3\}$.

It remains to show that $k\ne3$. Let $\beta_n:=D'(C_n,-1)$,
for every $n\geq 1$, where $D'(C_n,x)$ is the derivative of $D(C_n,x)$ with respect to $x$. Then by the recursive
 formula given in Lemma
\ref{cycle} we conclude that for every $n$, $n\geq 4$,
$\beta_n=-(\alpha_n+\beta_{n-1}+\beta_{n-2}+\beta_{n-3})$. Now,
by induction on $n$ and using Lemma \ref{-1}, we have:
\begin{equation}\label{lemma5}
   \beta_n=\left\{
  \begin{array}{ll}
    -n, & \hbox{if $n\equiv0\pmod4$;} \\
    n, & \hbox{if $n\equiv1\pmod4$;} \\
       0, & \hbox{otherwise.}
  \end{array}
\right.
\end{equation}

 Let $\theta_n:=D''(C_n,-1)$, for every $n$, $n\geq 1$. By Lemma \ref{cycle}, we conclude that
 for every $n\geq 4$,
$\theta_n=-2\alpha_n-2\beta_{n}-(\theta_{n-1}+\theta_{n-2}+\theta_{n-3}).$
Now, by induction on $n$, using Lemma \ref{-1} and
relation (\ref{lemma5}), we obtain the following:
\begin{equation}\label{lemma6}
\theta_n=\left\{
  \begin{array}{ll}
   n(n-4)/2, & \hbox{if $n\equiv0\pmod4$;} \\
    -n(n-1)/2, & \hbox{if $n\equiv1\pmod4$;} \\
     n(n+2)/4, & \hbox{if $n\equiv2\pmod4$;}\\
 0, & \hbox{if $n\equiv3\pmod4$.}
  \end{array}
\right.
\end{equation}

Now, let $k=3$. Thus
\begin{equation}\label{k=3}
D(C_n,x)=\prod_{i=1}^3 D(C_{n_i},x).
\end{equation}
By putting $x=-1$ in relation (\ref{k=3}), we find that  $\alpha_n=\alpha_{n_1}\alpha_{n_2}\alpha_{n_3}$.
 Since $n=n_1+n_2+n_3$, by Lemma \ref{-1}, ten cases can be considered:
\begin{enumerate}
  \item[1)]~~$n\equiv 0\pmod4$, $n_1\equiv 0\pmod4$,
$n_2\equiv 1\pmod4$, $n_3\equiv 3\pmod4$;
  \item[2)] ~ $n\equiv 0\pmod4$, $n_1\equiv 0\pmod4$,
$n_2\equiv 2\pmod4$, $n_3\equiv 2\pmod4$;
 \item[3)] ~ $n\equiv 1\pmod4$, $n_1\equiv 1\pmod4$,
$n_2\equiv 1\pmod4$, $n_3\equiv 3\pmod4$;
 \item[4)] ~ $n\equiv 1\pmod4$, $n_1\equiv 1\pmod4$,
$n_2\equiv 2\pmod4$, $n_3\equiv 2\pmod4$;
 \item[5)] ~ $n\equiv 1\pmod4$, $n_1\equiv 3\pmod4$,
$n_2\equiv 3\pmod4$, $n_3\equiv 3\pmod4$;
   \item[6)] ~ $n\equiv 2\pmod4$, $n_1\equiv 1\pmod4$,
$n_2\equiv 2\pmod4$, $n_3\equiv 3\pmod4$;
\item[7)] ~ $n\equiv 2\pmod4$, $n_1\equiv 2\pmod4$,
$n_2\equiv 2\pmod4$, $n_3\equiv 2\pmod4$;
 \item[8)] ~ $n\equiv 3\pmod4$, $n_1\equiv 1\pmod4$,
$n_2\equiv 1\pmod4$, $n_3\equiv 1\pmod4$;
  \item[9)] ~ $n\equiv 3\pmod4$, $n_1\equiv 1\pmod4$,
$n_2\equiv 3\pmod4$, $n_3\equiv 3\pmod4$;
  \item[10)] ~ $n\equiv 3\pmod4$, $n_1\equiv 2\pmod4$,
$n_2\equiv 2\pmod4$, $n_3\equiv 3\pmod4$.

\end{enumerate}
 For instance, if Case 1 occurs, by derivative of two sides of the equality (\ref{k=3})
 and putting $x=-1$, we find that $\beta_n=\beta_{n_1}\alpha_{n_2}\alpha_{n_3}+\alpha_{n_1}\beta_{n_2}\alpha_{n_3}+\alpha_{n_1}\alpha_{n_2}\beta_{n_3}$.
  Now, by Lemma \ref{-1} and relation (\ref{lemma5}) we obtain that $n_3=2n_2$ which is
  impossible. Similarly, in cases 2, 3, 4, 5, 6, 8, and 9 we obtain
  a contradiction.

If the Case 7 occurs then, by the second derivative of two sides of
equality (\ref{k=3}) and putting $x=-1$, we conclude
that:$$\theta_n=\theta_{n_1}\alpha_{n_2}\alpha_{n_3}+\alpha_{n_1}\theta_{n_2}\alpha_{n_3}+\alpha_{n_1}\alpha_{n_2}\theta_{n_3}
+2\beta_{n_1}\beta_{n_2}\alpha_{n_3}+2\beta_{n_1}\beta_{n_3}\alpha_{n_2}+2\beta_{n_2}\beta_{n_3}\alpha_{n_1}.$$
Now, by Lemma \ref{-1}, and using relations (\ref{lemma5}) and
(\ref{lemma6}) we find that, $n_1n_2+n_1n_3+n_2n_3=0$ which is
impossible. Similarly, for case 10 we get a contradiction. Thus
$k=1$ and the proof is complete.}
\end{proof}

By the following lemma and Theorem \ref{theorem5}, the next corollary follows immediately.
\begin{lem} {\rm\cite[Corollary 2]{aap}} If $G$ is ${\cal D}$-unique,
then for every $m$, $m\ge1$, $G\vee K_m$ is ${\cal D}$-unique.
\end{lem}
\begin{cor} For every two positive integers $m$ and $n$, $K_m\vee C_n$ is  ${\cal
D}$-unique. In particular $W_n$ is ${\cal D}$-unique.
\end{cor}

\noindent{\bf Acknowledgements.} The research of the first author
was in part supported by a grant (No. 87050212) from school of
Mathematics, Institute for Research in Fundamental Sciences (IPM).

\end{document}